\PassOptionsToPackage{usenames, dvipsnames}{color}

\documentclass[12pt,letterpaper,english]{amsart}

\usepackage{amsmath,amsfonts,amssymb,mathrsfs,stackrel}
\usepackage[english]{babel}
\usepackage[utf8]{inputenc}
\usepackage[T1]{fontenc}
\usepackage{graphicx}
\usepackage{tabularx}
\usepackage{mathtools}
\usepackage{hyperref}
\usepackage{cleveref}
\usepackage{comment}
\usepackage[all]{xy}
\usepackage{xurl}
\numberwithin{equation}{subsection}
\usepackage{chngcntr}
\counterwithout{equation}{section}

\hypersetup{colorlinks,citecolor=blue,linktocpage}

\newcommand\cC{{\mathcal C}}
\newcommand\cD{{\mathcal D}}

\newcommand\cF{{\mathcal F}}

\newcommand\cI{{\mathcal I}}

\newcommand\cM{{\mathcal M}}

\newcommand\cO{{\mathcal O}}

\newcommand\cS{{\mathcal S}}
\newcommand\cT{{\mathcal T}}
\newcommand\cU{{\mathcal U}}

\newcommand\cX{{\mathcal X}}

\newcommand\bC{{\mathbb C}}

\newcommand\bF{{\mathbb F}}

\newcommand\N{{\mathbb N}}

\newcommand\bP{{\mathbb P}}
\newcommand\Q{{\mathbb Q}}
\newcommand\R{{\mathbb R}}

\newcommand\bV{{\mathbb V}}

\newcommand\Z{{\mathbb Z}}

\newcommand\bfG{{\mathbf G}}
\newcommand\bfH{{\mathbf H}}

\DeclareMathOperator{\GL}{GL}

\DeclareMathOperator{\Hom}{Hom}

\DeclareMathOperator{\PGL}{PGL}

\DeclareMathOperator{\rank}{rank}

\DeclareMathOperator{\der}{der}

\DeclareMathOperator{\ad}{ad}
\DeclareMathOperator{\fix}{fix}
\DeclareMathOperator{\disc}{disc}

\usepackage{amssymb,tikz}
\usepackage{tikz-cd}

\newtheorem{theorem}{Theorem}[section]

\newtheorem{corollary}[theorem]{Corollary}

\newtheorem{proposition}[theorem]{Proposition}

\theoremstyle{definition}

\newtheorem{assumption}[theorem]{Assumption}

\newtheorem{definition}[theorem]{Definition}
\newtheorem{example}[theorem]{Example}

\newtheorem{question}[theorem]{Question}
\newtheorem{remark}[theorem]{Remark}

\usepackage[usenames,dvipsnames]{color}

\usepackage{cleveref}

\newcommand{\alice}[1]{{\color{purple} \sf  Alice: [#1]}}

\title[Shafarevich's conjecture over function fields]{Shafarevich's conjecture for families of hypersurfaces over function fields}
\author{Philip Engel}
\address{Department of Mathematics, Statistics, and Computer Science, University of Illinois at Chicago, Chicago, IL, USA}
\email{pengel@uic.edu}
\author{Alice Lin}
\address{Department of Mathematics, Harvard University, 1 Oxford St, Cambridge, MA 02138, USA}
\email{alicelin@math.harvard.edu}

\author{Salim Tayou}
\address{Department of Mathematics, Dartmouth College, Kemeny Hall, Hanover, NH 03755, USA}
\email{salim.tayou@dartmouth.edu}
\date{\today}

\begin{document}

\begin{abstract}
Given a smooth quasi-projective complex algebraic variety $\mathcal{S}$, we prove that there are only finitely many Hodge-generic non-isotrivial families of smooth projective hypersurfaces over $\mathcal{S}$ of degree $d$ in $\mathbb{P}_{\bC}^{n+1}$. We prove that the finiteness is uniform in $\mathcal{S}$ and give examples where the result is sharp.

We also prove similar results for certain complete intersections in $\mathbb{P}_{\bC}^{n+1}$ of higher codimension and more generally for algebraic varieties whose moduli space admits a period map that satisfies the infinitesimal Torelli theorem. 
\end{abstract}

\maketitle

\setcounter{tocdepth}{1}
\tableofcontents

\section{Introduction}

In his famous ICM address \cite{shafarevich-ICM-1962}, Shafarevich asked the following question: 

\begin{question}
Given a number field $K$ and a finite set of places $T$ of $K$, are there only finitely many algebraic 
varieties of a certain type over $K$ with good reduction outside of $T$?
\end{question}

This question, now known as the Shafarevich conjecture over number fields, is a central problem in number theory. For example, Faltings proved the Shafarevich conjecture for abelian varieties over number fields en route to proving Mordell's conjecture in \cite{faltings-finiteness}.
In the same ICM address, Shafarevich also proposed the following complex-geometric version of the question for families of curves.

\begin{question}\label{question:shafarevich-curves-function-field}
    Fix an integer $g \ge 1$, a complex projective curve $\cS$, and a finite set $\cT$ of points in $\cS$. 
    Are there only finitely many non-isotrivial families of smooth projective genus $g$ curves over $\cS \backslash \cT$?
\end{question}

One can also formulate the question more generally as what is often referred to as the \textit{geometric Shafarevich conjecture}:

\begin{question} \label{question:shafarevich-geom}
    Let $\cS$ be a smooth quasiprojective variety over $\bC$. Are there only finitely many smooth projective non-isotrivial families of algebraic varieties of a given type over $\cS$?
\end{question}

\Cref{question:shafarevich-curves-function-field} was resolved by Parshin and Arakelov \cite{parsin-curves, arakelov}, and \Cref{question:shafarevich-geom} was answered by Faltings for families of abelian varieties over complex curves \cite{faltings1983arakelov}. It has also been studied in other situations in \cite{peters_arakelovs_1986, saito-AVs1993, bedulev-viehweg-shafarevich-surfaces-generaltype, liu_shafarevichs_2005, liu-todorov-yau-zuo-calabi-yau, kovacs_logarithmic_2002, kovacs-lieblich,javanpeykar-litt-integral-points} and more recently in \cite{ascher-taji, zuo-recent-paper}. We refer to the introduction of \cite{kovacs-lieblich} for more references and detailed historical background. Throughout the paper, the finite set of such families could very well be empty.  

In this paper, we prove the following result.

\begin{theorem}
    \label{thm:main-hypersurfaces}
    Let $\cS$ be a smooth quasiprojective variety  over $\bC$ and let $n\ge 1, d \geq 3$. 
    Then there are only finitely many non-isotrivial Hodge-generic\footnote{With respect to the primitive cohomology. One can also replace this condition by a simple condition on the algebraic monodromy of the family, see \Cref{monodromy-is-enough}.} families of smooth projective hypersurfaces of degree $d$ and dimension $n$ over $\cS$.
\end{theorem}

We refer to \Cref{def:Hodge-generic} for a definition of Hodge-generic, which in this case we take to be with respect to the moduli space of hypersurfaces of degree $d$ and dimension $n$.

When $\cS$ is a curve\footnote{The uniformity holds more generally for fibers of any smooth family $\cC\rightarrow \cT$, see \Cref{subsection:proofofmaintheorem}.}, we have furthermore the following uniform finiteness result.

\begin{theorem}
    \label{thm:uniform-curves}
    There exists a constant $N=N(d,n,g,r)$, such that 
    for any smooth quasiprojective curve $\mathcal{S}$ of genus $g$ with $r$ punctures, 
    the number of non-isotrivial Hodge-generic families of hypersurfaces 
    of degree $d$ and dimension $n$ is at most $N$.
\end{theorem}

One may define a discriminant for a family of hypersurfaces over a quasiprojective complex curve $\cS$, see \Cref{def:discriminant}. We then prove the following.

\begin{corollary}
\label{cor:discriminant}
    Let $\mathcal{S}$ be a smooth quasi-projective
    curve over $\bC$. Then there are only finitely many smooth 
    projective non-isotrivial Hodge-generic families of $n$-dimensional hypersurfaces of arbitrary degree $d$ and bounded discriminant over $\mathcal{S}$.
\end{corollary}

We remark that the assumption on the family being Hodge-generic is mild in light  of recent results of \cite{baldi2024distribution}.

Indeed, the authors proved that for $n\geq 3,d\geq 5$, $(n,d)\neq (4,5) $, there exists a Zariski open subset $U$ \footnote{The set $U$ is the complement of certain unspecified maximal components of the Hodge locus.} in the moduli space of hypersurfaces $\mathcal{U}_{n,d}$ such that any variety $\cS$ intersecting $U$ nontrivially\footnote{I.e., the intersection is nonempty.} is Hodge-generic. On the other hand, Saito--Zucker \cite[Section 8]{saito-zucker-non-rigid-k3} and Viehweg--Zuo \cite[Theorem 0.3]{viehweg-zuo-hypersurfaces} have given examples of non-Hodge generic families 
of K3 surfaces\footnote{Which can be chosen to be quartic surfaces.} and hypersurfaces with $d\geq n+1$ which deform non-trivially in positive dimensional families. Hence the assumption on the Mumford--Tate group in \Cref{thm:main-hypersurfaces} is necessary in general. 

Finally, we also show the following result.
\begin{theorem}
    \label{thm:main-complete-intersection}
    Let $\cS$ be a smooth quasiprojective variety  over $\bC$ and let $n\ge 1$, $\underline{d}=(d_1,\ldots,d_r)$ integers with $d_i\geq 2$ for all $i$ and $r\geq 2$. %excluding $(n,d)=(1,2),(1,3)$, 
    Then there are only finitely many Hodge-generic families of smooth complete intersections of degree $\underline d$ and dimension $n$ over $\cS$, except possibly when $r=2$, $d_1=d_2=2$ and $n$ is even. The finiteness is furthermore uniform when $\cS$ varies in smooth families.
\end{theorem}

\begin{remark}
 Theorems \ref{thm:main-hypersurfaces} and \ref{thm:main-complete-intersection} generalize to hypersurfaces in Severi--Brauer schemes over $\cS$. 
\end{remark}

\subsection{Main contribution and earlier work}\label{main-contributions}

The approach of Parshin and Arakelov for \Cref{question:shafarevich-curves-function-field} uses two ingredients: boundedness and rigidity. 
Given $\mathcal{M}$ a separated Deligne--Mumford moduli stack which is given as the fine moduli space of algebraic varieties, and $\mathcal{S}$ a smooth quasi-projective variety, the space of maps $\mathrm{Hom}(\mathcal{S},\mathcal{M})$ corresponds to families of algebraic varieties over $\mathcal{S}$ that are parameterized by $\mathcal{M}$.

Thus, one would like to prove that the mapping space $\mathrm{Hom}(\mathcal{S},\mathcal{M})$ is of finite type (boundedness) and zero-dimensional (rigidity).
Kov\'acs and Lieblich have shown in \cite{kovacs-lieblich} that boundedness holds when $\mathcal{M}$ is a weakly bounded compactifiable Deligne--Mumford stack, see \Cref{subsection:boundednessoffamilies}, 
and that rigidity holds if the stack satisfies an ``infinitesimal rigidity'' condition \cite[Theorems 1.7, 5.2]{kovacs-lieblich}. 

The main observation of this paper is to recast the two problems in the framework of Hodge theory: the boundedness problem is solved using the recent resolution of the conjecture of Griffiths by Brunebarbe-Bakker-Tsimerman \cite{bakker-brunebarbe-tsimerman}, and we give a sufficient condition for rigidity in Hodge-theoretic terms and in situations where infinitesimal Torelli holds, directly inspired by earlier results of Peters, see \cite{peters_arakelovs_1986,peters-compositio-rigidity}. When $\cM$ is a variety, boundedness has also been previously obtained in \cite[Theorem 4.2]{javanpeykar-litt-integral-points}.

Our main theorem is the following.

\begin{theorem}
\label{main-theorem}
    Let $\mathcal{S}$ be a smooth quasiprojective algebraic variety and $\mathcal{M}$ be a separated, finite type, smooth Deligne--Mumford stack which is a fine moduli space for polarized projective algebraic varieties. 
    Assume that $\cM$ admits a period map that satisfies the infinitesimal Torelli theorem, and let $H$ be its algebraic monodromy group.
    
    Then Shafarevich finiteness holds for families over $\mathcal{S}$ parameterized by $\mathcal{M}$\footnote{i.e., elements of $\mathrm{Hom}(\cS,\cM)$.} and whose algebraic monodromy group is equal to $H$.
    
    Moreover, for a smooth family $\cC\rightarrow \cT$, the number of families over $\cC_t$ for $t\in T$ satisfying the monodromy condition above is uniformly bounded in $t\in T$. % is bounded uniformly when $\mathcal{S}$ is the fiber of a smooth family $\cC\rightarrow \cT$.
\end{theorem}

Independently, Chen, Hu, Sun, and Zuo proved similar results in \cite{zuo-recent-paper}.

\begin{remark}
    Corollary 7.1 in \cite{bakker-brunebarbe-tsimerman} shows that a complex analytic space 
    fibered in varieties parameterized by $\mathcal{M}$ over a quasiprojective algebraic variety
    is actually algebraic. Hence all the previous results hold true if we work 
    with \emph{analytic families over an algebraic base $\cS$}.
\end{remark}

\subsection{Strategy of the proof} \label{subsection:strategy}

As explained in \Cref{main-contributions}, the proof of Theorem \ref{thm:main-hypersurfaces} has two ingredients: 
we first prove that the space of maps $\mathrm{Hom}(\mathcal S,\mathcal M)$ is of finite type over $\bC$ using a method inspired from Kov\'acs and Lieblich. 
%result to show that $\mathrm{Hom}(\mathcal S,\mathcal M)$ is a finite type stack.. 
We first show that $\cM$ is weakly bounded and compactifiable using the Arakelov inequality proved by Peters \cite{peters_arakelovs_1986}, see also Brunebarbe \cite{brunebarbe-arakelov} for a generalization, and the ampleness of the Griffiths line bundle on the coarse moduli space associated to $\mathcal{M}$ established recently in \cite{bakker-brunebarbe-tsimerman}. As for the rigidity, we use a result of Peters \cite{peters-compositio-rigidity} to analyze deformation of period maps and relate that to deformation of maps to $\mathcal{M}$. When the algebraic monodromy is maximal, i.e., equal to the monodromy of the universal family, we prove that these maps are rigid and thus $\mathrm{Hom}(\mathcal S,\mathcal M)$ is zero dimensional at those points.

\subsection{Organization of the paper} 
In \textsection \ref{section:boundedness}, we recall background from Hodge theory (\textsection \ref{subsection:hodgetheorybackground}) and state the Arakelov inequality for curves (\textsection \ref{subsection:arakelov}), which we then use to prove boundedness of the space of maps from an algebraic variety to a Deligne-Mumford stack which admits an immersive period map (\textsection \ref{subsection:boundednessoffamilies}).

In \textsection \ref{section:rigidity}, we prove rigidity of certain components in this space of maps by relating its tangent space to that of the space of period maps (\textsection \ref{subsection:rigidityoffamilies}), which can be described in Hodge theoretic terms (\textsection \ref{subsection:rigidityofperiodmaps}). We can then prove \Cref{main-theorem} in \textsection\ref{subsection:proofofmaintheorem}. 

In \textsection \ref{section:infinitesimaltorelli}, we survey known infinitesimal Torelli results so that \Cref{thm:main-hypersurfaces} follows from \Cref{main-theorem}. We then define the notion of discriminant for a family of smooth projective hypersurfaces over a quasiprojective complex curve and prove \Cref{cor:discriminant} (\textsection \ref{subsection:discriminants}).

\subsection{Acknowledgments}
We thank H\'el\`ene Esnault for her interest in our work and for helpful suggestions, Aaron Landesman for suggesting Remark 1.8, Ariyan Javanpeykar and Chris Peters for helpful comments and suggestions, and Matt Kerr for pointing out an error in Section 3.3. We thank the referee for valuable suggestions. P.E. ~ was supported by DMS-2201221. S.T.~was supported by NSF grant DMS-2302388 and NSF grant DMS-2503815. A.L.~was supported by the National Science Foundation Graduate Research Fellowship Program under Grant No.~DGE 2140743. %Any opinions, findings,and conclusions or recommendations expressed in this material are those of the authorsand do not necessarily reflect the views of the National Science Foundation.

\section{Boundedness results}\label{section:boundedness}

\subsection{Hodge-theoretic background} \label{subsection:hodgetheorybackground}
% notation of $bV$, Hodge filtration, polarization
% Griffiths transversality 
% Mumford Tate group is $G$
% Generic MT group defn
% What it means to be Hodge-generic
For more details on this section, we refer to 
\cite{voisin,green-griffiths-kerr,klingler-survey}. 

Let $\mathcal{S}$ be a smooth complex algebraic variety, 
and let $\bV=\{\bV_\Z,\cF^{\bullet},\psi\}$ be a polarized 
variation of Hodge structure of weight $k$ over $\mathcal{S}$. That is, $\bV_\Z$ is a local system of free finite rank $\Z$-modules, and
$$
    \cF^{\bullet}=\cF^k\subset\ldots\subset\cF^0=
    \bV_\Z\otimes_\Z\mathcal{O}_{\mathcal{S}}
$$
is a filtration by holomorphic subvector bundles of the flat bundle $(\cF^0,\nabla)$ associated to the local system $\bV_\Z$ by the Riemann--Hilbert correspondence, and $\psi:\bV_\Z\times \bV_\Z\rightarrow \underline{\Z}$ is a flat bilinear form that is a polarization of the pure Hodge structure $(\bV_s,\cF_s)$ for every point $s\in \mathcal{S}$. Moreover, $(\cF^0,\nabla)$ satisfies Griffiths' transversality: for all $1 \le p \le k$,
$$
    \nabla(\cF^p)\subset \cF^{p-1}\otimes \Omega^1_\cS~.
$$

For every point $s\in \mathcal{S}$, let $G_s$ be the Mumford--Tate group of the Hodge structure $(\bV_s,\cF^\bullet_s,\psi)$. It is the reductive algebraic subgroup over $\Q$ contained in $\GL(\mathbb V_{\mathbb Q, s})$ which fixes all Hodge tensors of $\mathbb{V}_{\Q,s}$. For a very general point $s$ in $\mathcal{S}$, it is isomorphic to a fixed group $G$, called the generic Mumford--Tate group of the variation. 

\begin{definition}\label{def:Hodge-generic}
    If $G$ is the generic Mumford--Tate group of a $\Z$-PVHS $\bV$ on $\cS$, and $\cS' \subset \cS$ is a subvariety, then we say that $\cS'$ is \emph{Hodge generic} in $\cS$, %if the restriction $\bV|_{\cS'}$ to $\cS'$ is \emph{Hodge-generic} 
    if the generic Mumford--Tate group of $\bV|_{\cS'}$ is also $G$. 
\end{definition}

% Define monodromy group (over Q)
Let $H$ be the algebraic monodromy group of $\bV$ at $s$, defined as the connected component of the Zariski closure over $\Q$ of the image of the monodromy representation corresponding to the local system $\bV_\Z$: 
\[
    \rho:\pi_1(\mathcal{S},s)
    \rightarrow \mathrm{GL}(\bV_{\Q,s})~.
\] 
By a theorem of Deligne and Andr\'e \cite{hodge2,andre-fixed-part}, $H$ is a normal subgroup of $G^{\rm der}$, the derived group of $G$ for very general $s \in \cS(\bC)$.%, for an analytically dense set of points $s \in \cS(\bC)$.

%Fixing such a basepoint $s$, we have then a decomposition over $\Q$ of the adjoint groups: 
%\[
 %   G^{\rm ad}= H^{\rm ad}\times G'.
%\] 
Let $\cD$ be the Mumford--Tate domain associated to $G$. 
It can be identified with $\cD = G^{\ad}(\R)^+/U$, where $U\subset G$ is the compact 
subgroup that fixes the Hodge filtration at the basepoint
$s\in \cS$. 
%Let $U_H=U\cap H^{\ad}(\R)^+$ and 
%let $U_{G'}=U\cap G'^{\ad}(\R)^+$. 
%We get then a decomposition of Mumford--Tate domains: 
%$\cD=\cD_H\times \cD_{G'} $.

%\alice{TODO: Define the period domain and period map, write the $X_\Gamma$ notation, say that this itself has a filtration of vector bundles}

There exists an arithmetic subgroup $\Gamma \subset G(\Q)$ such that the image of the monodromy representation $\rho$ is contained in $\Gamma$  and the variation 
of Hodge structure $\bV$ is completely described by its 
holomorphic period map: 
\[\varphi:\mathcal{S}\rightarrow X_\Gamma:= \Gamma \backslash \cD~. \]

\begin{comment}
Furthermore, up to a finite index, the group $\Gamma$ decomposes as $ \Gamma_H\times \Gamma_{G'}$ where $\Gamma_H$, resp. $\Gamma_{G'}$, is an arithmetic subgroup of $\bfH(\Q)$, resp. $\bfG'(\Q)$. Hence, up to a finite \'etale cover, the period map has a constant projection on the second factor \alice{Come back to check this}
and hence decomposes as 
\[
\varphi:\mathcal{S}\rightarrow \Gamma_H \backslash \cD_H
\times \{t_{\mathcal{S}}\}\subseteq 
\Gamma_H \backslash \cD_H\times \Gamma_{G'} \backslash \cD_{G'}
=\Gamma \backslash \cD~.
\]
\alice{What is the purpose of the above section aside from introducing the period domain notation?}
\end{comment}

\begin{example}
    The most important examples of polarized variations of Hodge structure are those of geometric origin: for 
    $\mathcal X\xrightarrow{\pi} \mathcal{S}$ a smooth projective morphism, the primitive part of the $\Z$-local system $R^{k}\pi_*\underline{\Z}_\cX$ along with the Hodge filtration determined by the $R^{k}\pi_*(\Omega^{\geq \bullet}_{\cX/\mathcal{S}})$ and the Poincar\'e pairing determines a polarized variation of Hodge structure over $\mathcal{S}$.
    These types of variations 
    (or more generally, local systems which, when restricted to a Zariski open $\cS^\circ \subset \cS$, are a summand of the primitive cohomology of a smooth projective family $\cX^\circ \to \cS^\circ$)
    are said to be of \emph{geometric origin}.
\end{example}

\subsection{Arakelov inequality} \label{subsection:arakelov}

In this section, we recall Arakelov's inequality %, recently established by Brunebarbe \cite{brunebarbe-arakelov}, 
which will be crucial to establishing boundedness. For history and background on Arakelov inequalities, see \cite{viehweg-arakelov-inequalities,kovacs-taji_arakelov_2024}. %\alice{TODO: check these references}

Let $\mathcal{S}$ be a smooth quasi-projective curve and 
let $\bV$ be a polarized variation of Hodge structure on 
$\mathcal{S}$ with period map $\varphi: \mathcal{S}\rightarrow X_\Gamma$.

The Griffiths line bundle $L$ on $X_\Gamma$ is defined as:
\[
    L:= \bigotimes_{p\geq 0} \det \cF^{p}.
\]
%The Griffiths line bundle admits a hermitian metric, the \textit{Hodge metric}, whose curvature is positive definite along the Griffiths distribution. In particular, the pullback of $L$ to $\cS$ along the period map $\varphi$ has positive semi-definite curvature (positive definite if the period map $\varphi$ is immersive) \alice{check this, or give a reference. @Salim: Is mentioning the curvature necessary?--it's not necessary}. 

Let $\mathcal{S}\hookrightarrow \overline{\mathcal{S}}$ be a smooth compactification of $\cS$ and unipotent monodromies around the points in $\overline{\mathcal{S}}\backslash \mathcal{S}$ and let $\overline{L}$ be the Deligne extension to $\overline{\mathcal{S}}$ of the pullback of the Griffiths line bundle $\varphi^* L$. Then we have the following crucial inequality originally due to Arakelov \cite{arakelov} in the case of families of curves and generalized for arbitrary non-isotrivial variations of Hodge structure over $\mathcal{S}$, see \cite{peters-arxiv} and \cite[Corollary 1.10]{brunebarbe-arakelov}. When $\cS$ has higher dimension, see \cite{brunebarbe-arakelov}.%, it can also be deduced from \cite{peters-arxiv}.

%, or \cite[Proposition 3]{peters_arakelovs_1986}.

\begin{proposition}\label{prop:arakelov-curve}
    Let $\mathcal{S}$ be a smooth quasi-projective curve and let $\bV$ a $\Z$-polarized variation of Hodge structure over $\mathcal{S}$. Then the following inequality holds: \[
        \deg_{\overline{\mathcal{S}}}(\overline{L})\leq C(\bV)\cdot \left(-\chi(\mathcal{S})\right),
    \]
    where $\chi(\cS)$ is the topological Euler characteristic of $\cS$ and  the constant $C(\bV)$ only depends on the discrete invariants of $\bV$, that is, the Hodge numbers and their indices.
\end{proposition}

\begin{remark}\label{explicit-bound}
    By \cite[Remark 1.7]{brunebarbe-arakelov}, we have the estimate:
    \[
        C(\bV)\leq \frac{w^2}{2}\cdot \rank(\bV),
    \]
    where $w$ is the level \footnote{By definition, if $[a,b]$ is the smallest interval such that $h^{i,n-i} = 0$ for $i \notin [a,b]$,
then the level of $\bV$ is the integer $b-a$.} of the Hodge structure $\bV$. 
\end{remark}

\subsection{Boundedness of families of algebraic varieties over $\mathcal{S}$} \label{subsection:boundednessoffamilies}

\subsubsection{Summary of Kov\'acs--Lieblich work}
Let $\cM$ be a smooth separated Deligne--Mumford stack of finite type over $\bC$ which admits a coarse moduli space that we denote by $M$. Following \cite[\S 4]{kovacs-lieblich}, we say that:
\begin{enumerate}
    \item $\cM$ is compactifiable if there exists an open immersion $\cM\hookrightarrow\cM^\dagger$ into a DM stack $\cM^\dagger$ proper over $\bC$.
    \item  $\cM$ is coarsely compactified if the coarse moduli space $M$ admits a compactification $M\hookrightarrow M^{\dagger}$, where $M^{\dagger}$ is a proper algebraic space over $\bC$.
\end{enumerate}

Assume that $M$ admits a compactification $M\hookrightarrow M^\dagger$ given by an ample line bundle $L$ over $M$. Given a function $b:\N^2\rightarrow \Z$, we say that $\cM$ is weakly bounded  with respect to $M^\dagger$ and $L$ by $b$, if for every curve $C^{\circ}$ of genus $g$ with $d$ punctures, and every morphism $\xi:C^{\circ}\rightarrow M$ factoring through a morphism $C^{\circ}\rightarrow \cM$, we have:
\[\deg \xi_C^*(L)\leq b(g,d),\] 
where $C^{\circ}\hookrightarrow C$ is a smooth compactification  and $\xi_C:C\rightarrow M^\dagger$ is the extension of the morphism $\xi$ to $C$.

Finally, we say that $\cM$ is weakly bounded, if we can find $b$ and $L$ as above such that $\cM$ is weakly bounded with respect to $M^\dagger$ and $L$ by $b$.

We have then the following theorem: 
\begin{theorem}[Theorem 1.7 in \cite{kovacs-lieblich}]\label{boundedness-theorem}
    Let $\cM$ be a weakly bounded compactifiable Deligne--Mumford stack over $\bC$. Given a morphism of algebraic spaces $\cU\rightarrow \cT$ that is smooth at infinity\footnote{See \cite[Definition 2.1]{kovacs-lieblich} for a definition}, there exists an integer $N$ such that for every geometric point $t\rightarrow \cT$, the number of deformation types of morphisms $\cU_t\rightarrow \cM$ is finite and bounded above by $N$.
\end{theorem}
Notice that if the morphism $\cU\rightarrow \cT$ is smooth, then in particular it is smooth at infinity.% in the sense of \cite[Definition 2.1]{kovacs-lieblich}.

%which is a fine moduli space for a universal family of smooth projective algebraic varieties $\cX \to \cM$. We denote by $M$ the associated coarse moduli space. We can use the results of \cite{kovacs-lieblich} and \cite{bakker-brunebarbe-tsimerman} to show the boundedness of the morphisms $\Hom(\cS, \cM)$ for a given $\bC$-scheme $\cS$. 
\subsubsection{Deligne--Mumford stacks with immersive period maps}
In this section, we make the following additional assumption. 
\begin{assumption} \label{c:torelli}
    We assume that $\cM$ supports a polarized variation of Hodge structure $\bV$ of weight $k$ %of geometric origin
    and for which the period map satisfies an infinitesimal Torelli theorem, i.e.~the period map is immersive.
\end{assumption}

\begin{remark}
    The previous assumption is restrictive on $\cM$. Indeed Koll\'ar has constructed examples \cite{kollar-non-quasi-projective} of moduli spaces of algebraic varieties which do not admit immersive period maps and it is unknown in general what conditions guarantee that  the Shafarevich conjecture holds for the algebraic varieties parameterized by these families. However, it will be satisfied in all the cases that we study in this paper.
\end{remark}

%The following theorem of \cite{bakker-brunebarbe-tsimerman} answering an earlier conjecture by Griffiths \cite{griffithsIII} will be crucial for us.

\begin{theorem}
     The stack $\cM$ is weakly bounded and compactifiable. 
\end{theorem}
\begin{proof}
By \cite[Corollary 7.3]{bakker-brunebarbe-tsimerman}, the coarse moduli space $M$ is a quasiprojective algebraic variety over $\bC$ and the pullback of the Griffiths line bundle $L$ to $M$ is ample and determines a compactification $M\hookrightarrow M^\dagger$. Hence by \cite[Theorem 4.4]{kresch-DMstacks}, $\cM$ is a global quotient of a quasiprojective scheme by a linear algebraic group. This in turn implies that $\cM$ is compactifiable by \cite[Lemma 4.2]{kovacs-lieblich}, which relies on the fact that $\cM$ admits a locally closed embedding into a smooth proper DM stack with projective coarse space as shown in \cite[Theorem 5.3, (iii)]{kresch-DMstacks}.

Let $b:\N^2\rightarrow \Z$ be the function given by \[b(g,d)=\frac{w^2}{2}\cdot \rank(\bV)\cdot(2g+d-2)~.\] Let $C^{\circ}\subset C$ be a curve of genus $g$ with $d$ punctures. We have by \Cref{prop:arakelov-curve} and \Cref{explicit-bound}: 
\[\deg_{C}(\overline{L})\leq b(g,d),\]
which shows that $\cM$ is weakly bounded with respect to $L$ and $M^{\dagger}$ by $b$.
\end{proof}

%\begin{theorem}\label{bbt}
 %   Let $\bV$ be polarized variation of Hodge structure on smooth quasi-projective algebraic variety $\mathcal X$ and let $\varphi:\mathcal X\rightarrow \Gamma\backslash \cD$ be the corresponding period map. Then $\varphi$ admits a factorization as $\mathcal X \xrightarrow{f} \mathcal Z \xhookrightarrow{\iota} \Gamma \backslash \cD$, where $\mathcal Z$ is an algebraic space, $f$ is a dominant map of finite type algebraic spaces, $\iota$ is a closed immersion of complex analytic spaces. Moreover, the Griffiths line bundle $L$ is ample when restricted to $\mathcal Z$, so that $\mathcal Z$ is a quasiprojective algebraic variety. 
%\end{theorem}

%By the previous theorem and \cite[Corollary 7.3]{bakker-brunebarbe-tsimerman}, the coarse moduli space $M$ is a quasiprojective variety over $\bC$ and the pullback of the Griffiths line bundle to $M$ is ample. 
% Note to self: the map f in the BBT theorem will have to be quasifinite, hence quasi affine, so the pullback of an ample line bundle is ample. SP Tag 0892
%Using the Arakelov inequality by Brunebarbe in  applied to curves together with the ampleness of the Griffiths line bundle on $M$, one can verify that the stack $\mathcal{M}$ is weakly bounded according to \cite[Definitions 4.4, 4.5]{kovacs-lieblich}.

%Moreover, our assumptions on $\cM$ imply that $\cM$ is a global quotient of a quasiprojective scheme by a linear algebraic group \cite[Theorem 4.4]{kresch-DMstacks}. This in turn implies $\cM$ is compactifiable by \cite[Lemma 4.2]{kovacs-lieblich}.

\begin{corollary}
Given an algebraic variety $\cS$, there are finitely many deformation types of morphisms $\cS \to \cM$, in particular, the complex analytic orbifold $\Hom(\cS, \cM)(\bC)$ has finitely many connected components. 
\end{corollary}
\begin{proof}
    Since $\cM$ is weakly bounded and compactifiable, then by \Cref{boundedness-theorem}, there are finitely many deformation types of morphisms from $\cS$ to $\cM$, hence $\Hom(\cS, \cM)$ is a finite type stack over $\bC$, and hence the complex analytic orbifold $\Hom(\cS, \cM)(\bC)$ has finitely many connected components. %Moreover, by choosing $C_{g,d}\rightarrow M_{g,d}$ to be the universal family, we conclude that there exists $N$ such for any $(g,d)$-curve $\cS$, the number of deformation types of morphisms $\cS$ to $\cM$ is bounded by $N$. 
\end{proof}

%In particular, this gives us boundedness of the mapping space $\Hom(\cS, \cM)$. 

\section{Rigidity results}\label{section:rigidity}

In this section, we prove that $\mathrm{Hom}(\mathcal{S},\mathcal{M})$ is zero-dimensional at the $\bC$-points with big monodromy, using rigidity results of period maps. 

\subsection{Rigidity of period maps}\label{subsection:rigidityofperiodmaps}
Let $\bV$ be $\Z$-PVHS over a smooth quasiprojective algebraic variety $\mathcal{S}$ over $\bC$, let $G$ be the generic Mumford--Tate group, and let $H$ be the monodromy group. Let \[\varphi:\mathcal{S}\rightarrow X_\Gamma:=\Gamma\backslash \cD\] be the corresponding period map. Let $T^hX_\Gamma\subseteq TX_\Gamma$ denote the Griffiths distribution. Then Griffiths transversality implies that $\varphi$ determines a point in $ \Hom^{||}(\mathcal{S},X_\Gamma)$, the space of holomorphic maps that are tangent to $T^hX_\Gamma$. By \cite[Corollary 2.6]{peters-compositio-rigidity}, the set $\Hom^{||}(\mathcal{S},X_\Gamma)$ has the structure of a finite-dimensional analytic variety\footnote{with possibly infinitely many irreducible components} and its Zariski tangent space at $\varphi$ is described by the following proposition, see \cite[(10.6)]{kerr-pearlstein}.\footnote{See also \cite[Theorem 3.2]{peters-compositio-rigidity} for earlier work.}
\begin{proposition}\label{zariski-tangent-space}
    The Zariski tangent space of $\mathrm{Hom}^{||}(\mathcal{S}, X_\Gamma)$ at $\varphi$ is given by the $(-1,1)$ part of the  Hodge structure $\mathfrak{g}_\bC^{\pi_1(\mathcal{S},s)}$, where $\mathfrak{g}\subset \mathrm{End}(V)$ is the Lie algebra of $G$.%the fiber of $\bV_{\bC}$ at a point $s\in \cS$,  %and $\bV$ is the $\Z$-local system giving rise to the period map $\varphi$. %\alice{@Salim: the period map gives a $\Z$-PVHS and not a $\Q$-PVHS right?--Yes, rather the Z-PVHS gives the period map.}
\end{proposition}

\begin{remark}
    Since the geometric monodromy group $H$ is the connected component of the closure of $\rho(\pi_1(\cS))$ in the Zariski topology, we have \[\mathfrak{g}^{\pi_1(\mathcal{S})} = \mathfrak{g}^{\overline{\pi_1(\mathcal{S})}} \subseteq \mathfrak{g}^H~.\]
\end{remark}

\begin{proposition} \label{prop:Hrigidity}
    Keep the same notation as above. If $H=G^{\der}$, the derived group of $G$, then the period map of $\cS$ is horizontally rigid.
\end{proposition}
\begin{proof}
    By \Cref{zariski-tangent-space}, the Zariski tangent space of horizontal deformations of $\varphi$ is equal to the $(-1,1)$ part of $
    \mathfrak{g}^{\pi_1(\mathcal{S})}$. 
    Since $G$ is reductive, the invariant part $
    \mathfrak{g}^{\pi_1(\mathcal{S})}$ is purely of type $(0,0)$, equal to the Lie algebra of its center, which is Hodge--Tate. Therefore, the $(-1,1)$ part is trivial.
    %Since $H$ acts absolutely irreducibly on $V$, by Schur's lemma, $\mathrm{End}(V_{\bC})^H=\bC$, hence the $(-1,1)$ part of the subspace $\mathrm{End}(V_{\bC})^{\pi_1(\mathcal{S})}$ is trivial, which concludes the proof. 
\end{proof}

\subsection{Rigidity of families parameterized by a DM stack}\label{subsection:rigidityoffamilies}

Let $\mathcal M$, $M$, and $\bV$ be as in \Cref{subsection:boundednessoffamilies}.
Let $f:\mathcal{S}\rightarrow \mathcal M$ be a non-constant morphism, let $\varphi: M \rightarrow X_\Gamma$ be the period map for the polarized variation of Hodge structure $\bV$, and let $\varphi \circ f:\mathcal{S} \rightarrow X_\Gamma$ be the composition.

The following proposition is a consequence of assuming the infinitesimal Torelli property. 

\begin{proposition}\label{prop:injectivetangentspace}
With Assumption \ref{c:torelli} and notation as above, the map on Zariski tangent spaces induced from the natural map \[\Hom(\mathcal{S},\mathcal M)\rightarrow \Hom^{||}(\mathcal{S},X_\Gamma)\] is injective. 
\end{proposition}
\begin{proof} 
    Fix an element $f \in \Hom(\mathcal{S},\mathcal M)$. 
  %  By \cite[Theorem ]{bakker-brunebarbe-tsimerman}, there exists a closed analytic subspace $\cZ$ of $\Gamma \backslash \cD$ such that $\cZ$ is a quasiprojective variety and the image of $\varphi: M \to \Gamma \backslash \cD$ factors as $M \xrightarrow{\tilde{\varphi}} \cZ  \xhookrightarrow{\iota}\Gamma \backslash \cD$. 
%  Since the map $f \mapsto \varphi \circ f$ factors through $\Hom(\cS, \cZ)$, 
We have an induced map on Zariski tangent spaces: %T_{\tilde{\varphi} \circ f } \Hom_{\bC}(S, \cZ)  \hookrightarrow
    \begin{align}\label{tangent-map}
        T_f \Hom_{\bC}(\cS, \cM) \to  T_{\varphi \circ f }\Hom^{||}_{\bC}(S,X_\Gamma).
    \end{align}
%    where the last term is possibly infinite-dimensional because the morphisms are of complex-analytic spaces. 
The Zariski tangent space of $\Hom_{\mathbb C} (\mathcal{S}, \cM)$ is given by \[\Hom_{\mathcal{O}_{\mathcal{S}}}(f^* \Omega_{\cM/\bC}, \cO_{\mathcal{S}})~,\] see \cite[Lemma 1, p.1294]{arakelov},  and the map on tangent spaces induced by composing with ${\varphi}$ is deduced in the diagram below by precomposing with the natural map $d {\varphi}: {\varphi}^* \Omega^h_{X_\Gamma/\bC} \to \Omega_{\cM/\mathbb C}$ 
\[
\begin{tikzcd}
{ T_f \Hom_{\mathbb C} (\mathcal{S}, \cM)} \arrow[r, equal]            & {\Hom_{\mathcal{O}_{\mathcal{S}}}(f^* \Omega_{\cM/\bC}, \cO_{\mathcal{S}})} \arrow[d]         \\
{ T_{{\varphi} \circ f} \Hom_{\mathbb C}^{||} (\mathcal{S}, X_{\Gamma})} \arrow[r, equal] & {\Hom_{\mathcal{O}_{\mathcal{S}}}(f^*{\varphi}^* \Omega^h_{X_{\Gamma} /\bC}, \cO_{\mathcal{S}})}
\end{tikzcd}
\]
where $\Omega^h_{X_\Gamma/\bC}$ is the quotient by the annihilator of the horizontal distribution on $X_\Gamma$. Assumption \ref{c:torelli} states that $\varphi$ is immersive, hence $d \varphi: {\varphi}^* \Omega_{X_\Gamma /\bC} \to \Omega_{\cM/\mathbb C}$ is surjective, and so we conclude the map (\ref{tangent-map}) is injective. 
\end{proof}

\begin{corollary}\label{zero-dim}
    The connected components of $\mathrm{Hom}(\mathcal{S},\mathcal M)$ corresponding to families with monodromy group equal to $G^{\der}$ are zero-dimensional.
\end{corollary}

\begin{proof}
Let $f:\mathcal{S}\rightarrow \mathcal M$ be a non-constant algebraic map. Let $\varphi : \cM \rightarrow X_\Gamma$ be the period map. By Assumption \ref{c:torelli}, $\varphi$ is immersive. The map $f\rightarrow \varphi\circ f$ sends $ \mathrm{Hom}(\mathcal{S},\mathcal M)$ to the sublocus $\Hom^{||}(\mathcal{S},X_\Gamma)$ of maps which are tangent to the Griffiths distribution. By \Cref{prop:injectivetangentspace}, we have that the differential of map $\varphi$ induces an injection of tangent spaces: 
\[
    T_f \Hom(\mathcal S, \mathcal M) \hookrightarrow T_{\varphi\circ f}\mathrm{Hom}^{||}(\mathcal{S},X_\Gamma).
\]

By \Cref{zariski-tangent-space}, the space  $T_{\varphi\circ f}\mathrm{Hom}^{||}(\mathcal{S},\Gamma\backslash \cD)$ is identified with the $(-1,1)$-part of the Hodge structure $\mathfrak{g}_\bC^{\pi_1(\cS)}$.
Now we may apply \Cref{prop:Hrigidity} to conclude that
\[
    T_{\varphi\circ f}\mathrm{Hom}^{||}(\mathcal{S},\Gamma\backslash \cD)=0,
\]
and so the Zariski tangent space to $\mathrm{Hom}(\mathcal{S},\mathcal M)$ at $f$ is zero. Hence the result.
\end{proof}

\subsection{Proof of \Cref{main-theorem}}\label{subsection:proofofmaintheorem}

Let $G$ be the generic Mumford-Tate group and let $H$ the monodromy group underlying the variation of Hodge structure on $\cM$. We have a decomposition of adjoint groups over $\Q$: $G^{\ad}=H^{\ad}\times H^{\fix}$, which induces a decomposition of period domains $\cD=\cD_H\times \cD_{fix}$. Moreover, the period map also decomposes, up to a finite \'etale cover of $\cM$, as: 
\[\Phi:\cM \rightarrow  \Gamma_H\backslash \cD_H \times \{*\}\]  where the first projection still satisfies infinitesimal Torelli and is rigid. Therefore, we can assume without loss of generality that $H=G^{\der}$. Let $f: \mathcal S \to \mathcal M$ be a moduli map corresponding to a family parametrized by $\cS$ and such that its monodromy group is equal to $H=G^{\der}$. By \Cref{prop:Hrigidity} and \Cref{prop:injectivetangentspace}, the map $f$ is rigid. Combining this rigidity result with the boundedness result of \Cref{subsection:boundednessoffamilies} concludes the proof of the finiteness proof of \Cref{main-theorem}.

If $\cS$ is a fiber of a family $\cC\rightarrow \cT$ smooth at infinity\footnote{See \cite[Definition 2.1]{kovacs-lieblich} for a definition}, then  \Cref{boundedness-theorem} implies that the finiteness is uniform in the fibers of the family $\cC\rightarrow \cT$. In particular, if $\mathcal M_{g,d}$ is the moduli space of genus $g$ curves with $d$ punctures and $3$-level structure and $\cC_{g,d} \to \cM_{g,d}$ denote the universal family, then $\cC_{g,d} \to \cM_{g,d}$ is smooth, hence smooth at infinity in the sense of \cite[Definition 2.1]{kovacs-lieblich} and the uniform finiteness follows by \Cref{boundedness-theorem}.

\begin{remark}\label{monodromy-is-enough}
   If the generic Mumford--Tate group $G$ is simple, then any non-isotrivial Hodge generic family will automatically have maximal monodromy, by Deligne--Andr\'e's theorem. More generally, rigidity holds whenever the algebraic monodromy of $
\cS$ acts absolutely irreducibly on the Lie algebra of the monodromy of the universal family. % Notice also that the assumption of Hodge-genericity in the theorem statements serves primarily to ensure that the monodromy representation is absolutely irreducible, which was really the key assumption. 
\end{remark}

\section{Examples of Infinitesimal Torelli}\label{section:infinitesimaltorelli}

In this section, we state some known results in the literature about moduli spaces of hypersurfaces and smooth complete intersections so that we can deduce \Cref{thm:main-hypersurfaces} from \Cref{main-theorem}.

For $n \ge 1$ and $d \ge 3$, let $\cU_{n,d}$ denote the moduli space of smooth hypersurfaces of degree $d$ in $\mathbb{P}^{n+1}$. Then we have a universal family $\mathcal{H}_{n,d}\rightarrow \cU_{n,d}$ of smooth hypersurfaces and a period map $\varphi:\mathcal U_{n,d}(\bC)\rightarrow X_\Gamma$ where $\cD$ is the period domain parameterizing Hodge structures of weight $n$ on the middle cohomology of a hypersurface of degree $d$ and dimension $n$. %and with Hodge numbers $(h^{p,q})_{p+q=n-1}$. 
%Let $\varphi:U_{n,d}\rightarrow \Gamma\backslash \cD$ be the period map. 
\begin{proposition}\label{prop:hypersurfacetorelli}
  Let $d \ge 3$ and $n \ge 1$. %, if $(n,d) \neq(3,1)$,
   Then $\cU_{n,d}$ is a smooth separated Deligne--Mumford  stack over $\bC$. If moreover $(n,d) \neq (2,3)$, the period map $\varphi:\cU_{n,d}\rightarrow X_\Gamma$ is immersive. %(quadric and cubic curves, and cubic surfaes)
\end{proposition}
\begin{proof}
    For the construction and properties of the stack, see \cite[\S 3.1]{javanpeykar2017moduli} for an overview. See the below result of Flenner for the period map being immersive. 
\end{proof}

More generally, given $\underline{d}=(d_1,\ldots,d_r)$ integers $\geq 2$, let $\mathcal{I}_{\underline{d}}$ denote the moduli space of smooth complete intersection of dimension $n$ in $\mathbb{P}_\bC^{n+r}$. Then $\mathcal{I}_{\underline{d}}$ is a smooth DM separated stack by \cite[Th\'eor\`eme 1.6]{benoist-DM-complete-intersections} and we have the following by \cite[Theorem 3.1]{flenner}.
\begin{proposition}\label{Torelli} The period map for $\mathcal{I}_{\underline{d}}$ is immersive, except for:
    \begin{enumerate}
        \item cubic surfaces;
        \item even-dimensional smooth complete intersection of two quadrics.
    \end{enumerate}
\end{proposition}

For the moduli space of cubic surfaces, even though the period map of relative middle cohomology is trivial, the \emph{occult period map} constructed in \cite{allcock-carlson-toledo-2002} provides a period map which satisfies infinitesimal Torelli.

\begin{proof}
    [Proof of \Cref{thm:main-hypersurfaces} and \Cref{thm:main-complete-intersection}]
    By \Cref{prop:hypersurfacetorelli} and \Cref{Torelli}, we can take $\cM = \cU_{n,d}$ and $\cM=\cI_{\underline d}$ in \Cref{main-theorem}. By Beauville \cite[Th\'eor\`emes 2,4,5,6]{beauville-monodromie}\footnote{See also \cite[Corollary 18]{peters-steenbrink-monodromy}.}, the generic Mumford--Tate group is $\Q$-simple, therefore its derived group is equal to the algebraic monodromy and acts absolutely irreducibly. For cubic surfaces, we use the period map from \cite{allcock-carlson-toledo-2002} which satisfies infinitesimal Torelli and it results from \cite[Theorem 2.14]{allcock-carlson-toledo-2002} that the generic Mumford--Tate group of the family is equal to its monodromy.  This proves \Cref{thm:main-hypersurfaces} and \Cref{thm:main-complete-intersection}.
\end{proof}

\subsection{Discriminants and Proof of \Cref{cor:discriminant}}\label{subsection:discriminants}

Let $\cS$ be a quasi-projective curve over $\bC$. Fix $n,d \ge 1 $, and let $\cU_{n,d}$ be the moduli space of smooth hypersurfaces of degree $d$ in $\bP^{n+1}$. A non-isotrivial family of such hypersurfaces over $\cS$ corresponds to a nonconstant map $f: \cS \to \cU_{n,d}$. 
Recall that $\cU_{n,d} = [(\bP^{N(n,d)} -D_{n,d})/\PGL_{n+2}]$, where $\bP^{N(n,d)}$ is the parameter space for all degree $d$ homogeneous polynomials in $n+2$ variables, and $D_{n,d}$ is the locus of homogeneous polynomials corresponding to singular hypersurfaces (i.e.~the discriminant locus).

Any choice of coordinates on $\bP^{n+1}$ determines a lift $g: \cS \to \bP^{N(n,d)}$ of $f$. Any two lifts will differ by an $\PGL_{n+2}$-automorphism. %\alice{@Salim, is this correct? I'm treating $\bP^{N(n,d)}$ as a simply connected covering space of $\cU_{n,d}$. This is not entirely correct: a choice of a basis of $H^0(O(1))$ determines a lift of the map, so it's up to PGL. The automorphism can also depend of the point in the base...}
%\alice{We may have to discuss the choice of base point together. The main point is that I want the discriminant to be well-defined.}

Fix a smooth compactification $\cS \hookrightarrow \overline{\cS}$. There is a unique way to extend $g$ to $\overline g: \overline{\cS} \to \bP^{N(n,d)}$. By uniqueness of the extension, the extensions $\overline{g}_1, \overline{g}_2$ of any two choices of lifts $g_1,g_2$ differing by $\sigma \in \PGL_{n+2}$ will also differ by $\sigma$. 

Since the $\PGL_{n+2}$-action is invariant on $D_{n,d}$, the following defined quantity is independent of the choice of $g$ or $\overline{g}$.%, but might depend on the embedding $\cS \hookrightarrow \overline{\cS}$.

\begin{definition}\label{def:discriminant}
    Let $\cS$ be a quasiprojective complex curve with smooth compactification $\cS \hookrightarrow \overline{\cS}$. For a family of smooth hypersurfaces given by the map $f:\cS \to \cU_{n,d}$, define the \emph{discriminant} of this family as 
    \[
         \disc(f) := \deg_{\overline{\cS}} (\overline{g}^* \cO(D_{n,d})).
    \]
\end{definition}

\begin{remark}
    We can interpret the discriminant of the family $f:\cS \to \cU_{n,d}$ as counting how many singular fibers there are in the induced family of hypersurfaces over $\overline{\cS}$.
\end{remark}

By B\'ezout's theorem,
\begin{align}
    \disc(f) &= (\deg \overline g(\overline{\cS}))(\deg D_{n,d})(\deg \overline g) \label{eqn:discproduct}
\end{align}
where the first term on the right-hand side is the degree of the image of $\overline{g}$ as a closed subvariety of $\bP^{N(n,d)}$, the second term is the degree of $D_{n,d}$ as a subvariety of $\bP^{N(n,d)}$, and the third term is the degree of the finite morphism $\overline g$ onto its image. Using this formulation of the discriminant, we can prove \Cref{cor:discriminant}.

\begin{proof}
    [Proof of \Cref{cor:discriminant}]
    The degree of the discriminant locus $D_{n,d}$ is classically known to be $(n+2)(d-1)^n$ (see, e.g.~\cite[Prop.~7.4]{eisenbud20163264}). Using (\ref{eqn:discproduct}), we see that for a fixed base curve $\cS$ and dimension $n$ of hypersurfaces, there are only finitely many possible choices of $d$ to achieve a bounded discriminant. Moreover, for each such choice of $(n,d)$, \Cref{thm:main-hypersurfaces} implies that there are only finitely many Hodge-generic families of smooth projective hypersurfaces of degree $d$ and dimension $n$ over $\cS$, and we are done.
\end{proof}

\begin{remark}
    Using an analogous definition of discriminant in terms of intersection, a similar argument shows that there are only finitely many families of smooth projective hypersurfaces of dimension $n$ and arbitrary degree $d$ over a fixed quasiprojective $\bF_q$-curve $\cS$ for a given finite field $\bF_q$. In place of \Cref{thm:main-hypersurfaces}, one only needs to use the fact that for each $d$, there are only finitely many morphisms of bounded degree from $\overline{\cS} \to \bP^{N(n,d)}$ defined over $\bF_q$. 
\end{remark}

\bibliographystyle{alpha}
\bibliography{bibliographie}

\end{document}